\documentclass{amsart}

\usepackage{amssymb,amsfonts} 	
\usepackage{amsmath}            
\usepackage{amsthm}
\usepackage{bbm}

\newtheorem{theo}{Theorem}[section]
\newtheorem{lem}[theo]{Lemma}
\newtheorem{prop}[theo]{Proposition}

\numberwithin{equation}{section}

\theoremstyle{definition}

\numberwithin{equation}{section}

\begin{document}

\title{The sum of digits of polynomial values in arithmetic progressions}

\author{Thomas Stoll}
\address{Institut de Math\'ematiques de Luminy, Universit\'e d'Aix-Marseille, 13288 Marseille Cedex 9, France,}
\email{stoll@iml.univ-mrs.fr}



\date{\today}

\begin{abstract}
  Let $q, m\geq 2$ be integers with $(m,q-1)=1$. Denote by $s_q(n)$ the sum of digits of $n$
  in the $q$-ary digital expansion. Further let $p(x)\in \mathbb{Z}[x]$ be a
  polynomial of degree $h\geq 3$ with $p(\mathbb{N})\subset \mathbb{N}$. We show that there exist $C=C(q,m,p)>0$ and $N_0=N_0(q,m,p)\geq 1$, such that for all 
  $g\in\mathbb{Z}$ and all $N\geq N_0$,
  $$\#\{0\leq n< N: \quad s_q(p(n))\equiv g \bmod m\}\geq C N^{4/(3h+1)}.$$
  This is an improvement over the general lower bound given by Dartyge and Tenenbaum (2006), which is $C N^{2/h!}$.
\end{abstract}

\maketitle

\section{Introduction}

Let $q, m\geq 2$ be integers and denote by $s_q(n)$ the sum of digits of $n$ in the $q$-ary digital expansion of integers. 
In 1967/68, Gelfond~\cite{Ge67} proved that for nonnegative integers $a_1, a_0$ with $a_1\neq 0$, the sequence $\left(s_q(a_1n+a_0)\right)_{n\in \mathbb{N}}$ is well distributed in arithmetic progressions mod $m$, provided $(m,q-1)=1$. At the end of his paper, he posed the problem of finding
the distribution of $s_q$ in arithmetic progressions where the argument is restricted to values of polynomials of degree $\geq 2$. Recently, Mauduit and Rivat~\cite{MR09} answered Gelfond's question in the case of squares.

\begin{theo}[Mauduit \& Rivat (2009)]\label{thmMR09}
For any $q,m\geq 2$ there exists $\sigma_{q,m}>0$ such that for any $g\in\mathbb{Z}$, as $N \to \infty$,
$$\#\{0\leq n< N: \; s_q(n^2)\equiv g \bmod m\}= \frac{N}{m} \; Q(g,d)+O_{q,m}(N^{1-\sigma_{q,m}}),$$
where $d=(m,q-1)$ and $$Q(g,d)=\#\{0\leq n< d: \; n^2\equiv g \bmod d\}.$$
\end{theo}

The proof can be adapted to values of general quadratic polynomial instead of squares. We refer the reader to~\cite{MR10} and~\cite{MR09} for detailed references and further historical remarks. The case of polynomials of higher degree remains elusive so far. The Fourier-analytic approach, as put forward in~\cite{MR10} and~\cite{MR09}, seems not to yield results of the above strength. In a recent paper, Drmota, Mauduit and Rivat~\cite{DMR11} applied the Fourier-analytic method to show that well distribution in arithmetic progressions is obtained whenever $q$ is sufficiently large.

In the sequel, and unless otherwise stated, we write
$$p(x)=a_h x^h+\cdots+a_0$$ for an arbitrary, but fixed polynomial $p(x)\in\mathbb{Z}[x]$ of degree $h\geq 3$ with $p(\mathbb{N})\subset \mathbb{N}$.

\begin{theo}[Drmota, Mauduit \& Rivat (2011)]
  Let $$q\geq \exp(67 h^3 (\log h)^2)$$ be a sufficiently large prime number and suppose $(a_h,q)=1$. Then there exists $\sigma_{q,m}>0$ such that for any $g\in\mathbb{Z}$, as $N \to \infty$,
$$\#\{0\leq n< N: \; s_q(p(n))\equiv g \bmod m\}= \frac{N}{m} \; Q^\star(g,d)+O_{q,m,p}(N^{1-\sigma_{q,m}}),$$
where $d=(m,q-1)$ and $$Q^\star(g,d)=\#\{0\leq n< d: \; p(n)\equiv g \bmod d\}.$$
\end{theo}

It seems impossible to even find a single ``nice'' polynomial of degree $3$, say, that allows to conclude for well distribution in arithmetic progressions for small bases, let alone that the binary case $q=2$ is an emblematic case. Another line of attack to Gelfond's problem is to find lower bounds that are valid for all $q\geq 2$. Dartyge and Tenenbaum~\cite{DT06} provided such a general lower bound by a method of descent on the degree of the polynomial and the estimations obtained in~\cite{DT05}.

\begin{theo}[Dartyge \& Tenenbaum (2006)]
  Let $q,m \geq 2$ with $(m,q-1)=1$. Then there exist $C=C(q,m,p)>0$ and $N_0=N_0(q,m,p)\geq 1$, such that for all $g\in\mathbb{Z}$ and all $N\geq N_0$,
  $$ \#\{0\leq n< N: \quad s_q(p(n))\equiv g \bmod m\}\geq C N^{2/h!}.$$
\end{theo}

The aim of the present work is to improve this lower bound for all $h\geq 3$. More importantly, we get a substantial improvement of the bound as a function of $h$. The main result is as follows.\footnote{Gelfond's work and Theorem~\ref{thmMR09} give precise answers for linear and quadratic polynomials, so we do not include the cases $h=1, 2$ in our statement though our approach works without change.}

\begin{theo}\label{mtheo}
  Let $q,m \geq 2$ with $(m,q-1)=1$. Then there exist $C=C(q,m,p)>0$ and $N_0=N_0(q,m,p)\geq 1$, such that for all $g\in\mathbb{Z}$ and all $N\geq N_0$,
  $$\#\{0\leq n< N: \quad s_q(p(n))\equiv g \bmod m\}\geq C N^{4/(3h+1)}.$$
  Moreover,
  for monomials $p(x)=x^h$, $h\geq 3$, we can take
  \begin{align*}
      N_0 &=q^{3(2h+m)}\left(2hq^2\left(6q\right)^h\right)^{3h+1},\\
      C&=\left(16 hq^5 \left(6q\right)^h\cdot q^{(24h+12m)/(3h+1)}\right)^{-1}.
  \end{align*}
\end{theo}

The proof is inspired from the constructions used in~\cite{HLS11} and~\cite{HLS11pre} that were helpful in the proof of a conjecture of Stolarsky~\cite{St78} concerning the pointwise distribution of $s_q(p(n))$ versus $s_q(n)$. As a drawback of the method of proof, however, it seems impossible to completely eliminate the dependency on $h$ in the lower bound.

\section{Proof of Theorem~\ref{mtheo}}

Consider the polynomial 
\begin{equation}\label{goodt}
  t(x)=m_3 x^3+m_2 x^2-m_1 x+m_0,
\end{equation}  
where the parameters $m_0,m_1,m_2, m_3$ are positive real numbers that will be chosen later on in a suitable way. For all integers $l\geq 1$ we write
\begin{equation}\label{Tpower}
   T_l(x)=t(x)^l=\sum_{i=0}^{3l} c_i x^i
\end{equation}  
to denote its $l$-th power. (For the sake of simplicity we omit to mark the dependency on $l$ of the coefficients $c_i$.) The following technical result is the key in the proof of Theorem~\ref{mtheo}. It shows that, within a certain degree of uniformity in the parameters $m_i$, all coefficients \textit{but one} of $T_l(x)$ are positive.

\begin{lem}\label{lemtechn}
  For all integers $q\geq 2$, $l\geq 1$ and $m_0, m_1, m_2, m_3 \in \mathbb{R}^+$ with $$1\leq m_0, m_2, m_3 <q, \qquad 0<m_1<l^{-1} (6q)^{-l}$$ we have that $c_i >0$ for $i=0,2,3,\ldots,3l$ and $c_i<0$ for $i=1$. Moreover, for all $i$,
  \begin{equation}\label{coeffbound}
    |c_i|\leq (4q)^l.
  \end{equation} 
\end{lem}

\begin{proof}
  The coefficients of $T_l(x)$ in~(\ref{Tpower}) are clearly bounded above in absolute value by the corresponding coefficients of the polynomial $(qx^3+qx^2+qx+q)^l$. Since the sum of all coefficients of this polynomial is $(4q)^l$ and all coefficients are positive, each individual coefficient is bounded by $(4q)^l$. This proves~(\ref{coeffbound}).
  We now show the first part. To begin with, observe that $c_0=m_0^l>0$ and $c_1=-lm_1m_0^{l-1}$ which is negative for all $m_1>0$. Suppose now that $2\leq i\leq 3l$ and consider
  the coefficient of $x^i$ in
  \begin{equation}\label{Tl}
    T_l(x)=(m_3x^3+m_2x^2+m_0)^l+r(x),
  \end{equation}
  where 
  \begin{align*}  
    r(x) &=\sum_{j=1}^l \binom{l}{j}\left(-m_1x\right)^j \left(m_3x^3+m_2x^2+m_0\right)^{l-j}\\
    &=\sum_{j=1}^{3l-2} d_j x^j.
  \end{align*} 
  First, consider the first summand in~(\ref{Tl}).
   Since $m_0, m_2, m_3 \geq 1$ the coefficient of $x^i$ in the expansion of $(m_3x^3+m_2x^2+m_0)^l$ is $\geq 1$. Note also that all the powers $x^2, x^3, \ldots, x^{3l}$ appear in the expansion of this term due to the fact that every $i\geq 2$ allows at least one representation as $i=3i_1+2i_2$ with non-negative integers $i_1, i_2$. We now want to show that for sufficiently small $m_1>0$ the coefficient of $x^i$ in the first summand in~(\ref{Tl}) is dominant. To this end, we assume $m_1<1$ so that $m_1>m_1^j$ for $2\leq j \leq l$. Using $\binom{l}{j}<2^l$ and a similar reasoning as above we get that
 $$|d_j|< l 2^l m_1 (3q)^l= l\left(6q\right)^l m_1,\qquad 1\leq j\leq 3l-2.$$
 This means that if $m_1< l^{-1} (6q)^{-l}$ then the powers $x^2,\ldots, x^{3l}$ in the polynomial $T_l(x)$ indeed have positive coefficients. This finishes the proof.
\end{proof}

To proceed we recall the following splitting formulas for $s_q$ which are simple consequences of the $q$-additivity of the function $s_q$~(see~\cite{HLS11} for the proofs).

\begin{prop}\label{split}
  For $1\leq b<q^k$ and $a,k\geq 1$, we have
  \begin{align*}
    s_q(aq^k+b)&=s_q(a)+s_q(b),\\
    s_q(aq^k-b)&=s_q(a-1)+k(q-1)-s_q(b-1).	
  \end{align*} 
\end{prop}

We now turn to the proof of Theorem~\ref{mtheo}. To clarify the construction we consider first the simpler case of monomials, $$p(x)=x^h, \qquad h\geq 1.$$ 
(We here include the cases $h=1$ and $h=2$ because we will need them to deal with general polynomials with linear and quadratic terms.) Let $u\geq 1$ and multiply $t(x)$ in~(\ref{goodt}) by $q^{u-1}$. Lemma~\ref{lemtechn} then shows that for all \textit{integers} $m_0, m_1, m_2, m_3$ with 
\begin{equation}\label{admissible}
  q^{u-1}\leq m_0,m_2,m_3<q^u,\qquad 1\leq m_1 <q^u/(h q(6q)^h),
\end{equation}  
the polynomial $T_h(x)=(t(x))^h=p(t(x))$ has all positive (\textit{integral}) coefficients with the only exception of the coefficient of $x^1$ which is negative. Let $u$ be an integer such that 
\begin{equation}\label{uinterval}
  q^u\geq 2 hq(6q)^h
\end{equation}  
and let $k\in \mathbb{Z}$ be such that 
\begin{equation}\label{klarge}
  k>hu+2h.
\end{equation}
For all $u$ with~(\ref{uinterval}) the interval for $m_1$ in~(\ref{admissible}) is non-empty. Furthermore, relation~(\ref{klarge}) implies by~(\ref{coeffbound}) that $$q^k>q^{hu}\cdot q^{2h}\geq(4q^u)^h>|c_i|, \qquad \mbox{for all } i=0,1,\ldots,3h,$$ where $c_i$ here denotes the coefficient of $x^i$ in $T_h(x)$. Roughly speaking, the use of a large power of $q$ (i.e. $q^k$ with $k$ that satisfies~(\ref{klarge})) is motivated by the simple wish to split the digital structure of the $h$-power according to Proposition~\ref{split}. By doing so, we avoid to have to deal with carries when adding terms in the expansion in base $q$ since the appearing terms will not interfere. We also remark that this is the point where we get the dependency of $h$ in the lower bound of Theorem~\ref{mtheo}. 

Now, by $c_2, |c_1|\geq 1$ and the successive use of Proposition~\ref{split} we get
\begin{align}
  s_q(t(q^k)^h)&=s_q\left(\sum_{i=3}^{3h} c_i q^{ik} +c_2 q^{2k}-|c_1| q^{k}+c_0\right)\nonumber\\\nonumber
  &=s_q\left(\sum_{i=3}^{3h} c_i q^{(i-1)k} +c_2 q^{k}-|c_1|\right) +s_q(c_0)\\\nonumber
  &=s_q\left(\sum_{i=3}^{3h} c_i q^{(i-3)k}\right) +s_q(c_2-1)+k(q-1)-s_q(|c_1|-1)+s_q(c_0)\\\nonumber
  &=\sum_{i=3}^{3h} s_q(c_i)  +s_q(c_2-1)+k(q-1)-s_q(|c_1|-1)+s_q(c_0)\\
  &=k(q-1)+M,\label{Msimple}
\end{align}
where we write
$$M=\sum_{i=3}^{3h} s_q(c_i)  +s_q(c_2-1)-s_q(|c_1|-1)+s_q(c_0).$$
Note that $M$ is an integer that depends (in some rather obscure way) on the quantities $m_0, m_1, m_2, m_3$. Once we fix a quadruple
$(m_0,m_1,m_2,m_3)$ in the ranges~(\ref{admissible}), the quantity $M$ does not depend
on $k$ and is constant whenever $k$ satisfies~(\ref{klarge}). We now exploit the appearance of the single summand $k(q-1)$ in~(\ref{Msimple}). Since by assumption $(m,q-1)=1$, we find that
\begin{equation}\label{residuesystem}
  s_q(t(q^k)^h),\qquad \mbox{for } k=hu+2h+1,\; hu+2h+2,\; \ldots,\; hu+2h+m,
\end{equation}
runs through a complete set of residues mod $m$. Hence, in any case,
we hit a fixed arithmetic progression mod $m$ (which might be altered by $M$) for some $k$ with $hu+2h+1\leq k\leq hu+2h+m$.

Summing up, for $u$ with~(\ref{uinterval}) and by~(\ref{admissible}) we find at least 
\begin{equation}\label{atleast}
  (q^{u}-q^{u-1})^3(q^u/(hq(6q)^h)-1)\geq \frac{\left(1-1/q\right)^3}{2hq\left(6q\right)^h} \; q^{4u}
\end{equation}
integers $n$ that in turn by~(\ref{goodt}),~(\ref{admissible}),~(\ref{klarge}) and~(\ref{residuesystem}) are all smaller than
$$
q^u\cdot q^{3(hu+2h+m)}=q^{3(2h+m)} \cdot q^{u(3h+1)}$$
and satisfy $s_q(n^h)\equiv g \bmod m$ for fixed $g$ and $m$. By our construction and by choosing $k>hu+2h>u$ all these integers are distinct.
We denote 
$$ N_0 =N_0(q,m,p)=q^{3(2h+m)}\cdot q^{u_0(3h+1)},$$
where
$$ u_0 =\left \lceil \log_q\left(2hq(6q)^h\right)\right \rceil \leq \log_q\left(2hq^2(6q)^h\right).$$
Then for all $N\geq N_0$ we find $u\geq u_0$ with 
\begin{equation}\label{rangeNu}
  q^{3(2h+m)}\cdot q^{u(3h+1)}\leq N <q^{3(2h+m)}\cdot q^{(u+1)(3h+1)}.
\end{equation}
By~(\ref{atleast}) and~(\ref{rangeNu}), and using $(1-1/q)^3\geq 1/8$ for $q\geq 2$, we find at least
$$\frac{\left(1-1/q\right)^3}{2hq\left(6q\right)^h} \; q^{4u}\geq \left(16 hq^5 \left(6q\right)^h\cdot q^{(24h+12m)/(3h+1)}\right)^{-1}\; N^{4/(3h+1)}$$
integers $n $ with $0\leq n<N$ and $s_q(n^h)\equiv g \bmod m$. We therefore get the statement of Theorem~\ref{mtheo} for the case of monomials $p(x)=x^h$ with $h\geq 3$. The estimates are also valid for $h=1$ and $h=2$.

\bigskip

The general case of a polynomial $p(x)=a_h x^h+\cdots+a_0$ of degree $h\geq 3$ (or, more generally, of degree $h\geq 1$) follows easily from what we have already proven. Without loss of generality
we may assume that all coefficients $a_i$, $0\leq i\leq h$, are positive, since otherwise there exists $e=e(p)$ depending only on $p$ such that
$p(x+e)$ has all positive coefficients. Note that a finite translation can be dealt with choosing $C$ and $N_0$ appropriately in the statement. Since Lemma~\ref{lemtechn} holds for all $l\geq 1$ and all negative coefficients are found at the same power $x^1$, we have that the polynomial $p(t(x))$ has again all positive coefficients \textit{but one} where the negative coefficient again corresponds to the power $x^1$. It is then sufficient to suppose that
$$ k>hu+2h+\log_q \max_{0\leq i\leq h}{a_i}$$
in order to split the digital structure of $p(t(q^k))$. In fact, this implies that $$q^k>\left(\max\limits_{0\leq i\leq h}{a_i}\right)\cdot \left(4 q^u\right)^h,$$ and exactly the same reasoning as before yields $\gg_{q,p} q^{4u}$ distinct positive integers that are $\ll_{q,m,p} q^{u(3h+1)}$ and satisfy 
$s_q(p(n))\equiv g \bmod m$. This completes the proof of Theorem~\ref{mtheo}.

\subsection*{Acknowledgements}
This research was supported by the Agence Nationale de la Recherche, grant ANR-10-BLAN 0103 MUNUM.


\begin{thebibliography}{HD}

\normalsize
\baselineskip=17pt

\bibitem{Ge67}
A. O. Gelfond, Sur les nombres qui ont des propri\'et\'es additives et multiplicatives donn\'ees, \textit{Acta Arith.} \textbf{13} (1967/1968), 259--265.

\bibitem{DT05}
C. Dartyge, G. Tenenbaum, Sommes de chiffres de multiples d'entiers, 
\textit{Ann. Inst. Fourier (Grenoble)} \textbf{55} (2005), no. 7, 2423--2474.

\bibitem{DT06}
C. Dartyge, G. Tenenbaum, Congruences de sommes de chiffres de valeurs polynomiales,
\textit{Bull. London Math. Soc.} \textbf{38} (2006), no. 1, 61--69.

\bibitem{DMR11}
M. Drmota, C. Mauduit, J. Rivat, The sum of digits function of polynomial sequences, \textit{J. London Math. Soc.} \textbf{84} (2011), 81--102.

\bibitem{HLS11}
K. G. Hare, S. Laishram, T. Stoll, Stolarsky's conjecture and the sum of digits of polynomial values, \textit{Proc. Amer. Math. Soc.} \textbf{139} (2011), 39--49.

\bibitem{HLS11pre}
K. G. Hare, S. Laishram, T. Stoll, The sum of digits of $n$ and $n^2$, \textit{Int. J. Number Theory} (2011), to appear,\\
\texttt	{doi:10.1142/S1793042111004319}. 

\bibitem{MR10}
C. Mauduit,  J. Rivat, Sur un probl\`eme de Gelfond: la somme des chiffres des nombres premiers, \textit{Ann. of Math.} \textbf{171} (2010), 1591--1646. 

\bibitem{MR09}
C. Mauduit,  J. Rivat, La somme des chiffres des carr\'es, \textit{Acta Math.} \textbf{203} (2009), 107--148.

\bibitem{St78}
K. B. Stolarsky, The binary digits of a power, \textit{Proc. Amer. Math. Soc.} \textbf{71} (1978), 1--5.

\end{thebibliography}
\end{document}